\documentclass[reqno]{amsart}
\usepackage[utf8]{inputenc}
\usepackage[letterpaper, portrait, margin=1.2in]{geometry}
\usepackage[english]{babel}
\usepackage{amsthm}
\usepackage[utf8]{inputenc}
\usepackage{graphicx}
\usepackage{mathtools}
\usepackage{amsmath}
\usepackage{amssymb}
\usepackage{tabto}
\usepackage{enumitem}
\usepackage{tikz}
\usepackage{hyperref}
\usepackage{makecell}
\usepackage{comment}
\usetikzlibrary{calc}
\newtheorem*{theorem*}{Theorem}
\newtheorem{theorem}{Theorem}[section]
\newtheorem{corollary}[theorem]{Corollary}
\newtheorem{claim}[theorem]{Claim}

\newtheorem{lemma}[theorem]{Lemma}

\newtheorem{conjecture}[theorem]{Conjecture}
\theoremstyle{definition}

\newtheorem{counterexample}[theorem]{Counterexample}
\newtheorem{question}[theorem]{Question}

\DeclareMathOperator{\rev}{\mathrm{rev}}

\usepackage[skip=0.5\baselineskip]{caption}
\usepackage[toc,page]{appendix}

\newcommand{\Av}{\mathrm{Av}}
\newcommand{\usigma}{\underline{\sigma}}
\newcommand{\urevsigma}{\underline{\rev(\sigma)}}

\title{Periodic Points of Consecutive-Pattern-Avoiding Stack-Sorting Maps}
\author{Ilaria Seidel} \address{Department of Mathematics, Harvard University, Cambridge, MA 02138} \email{iseidel@college.harvard.edu}
\author{Nathan Sun} \address{Department of Applied Mathematics, Harvard University, Cambridge, MA 02138} \email{nsun@college.harvard.edu}
\date{August 10, 2023}

\begin{document}
\maketitle

\begin{abstract}
West's stack-sorting map involves a stack which avoids the permutation $21$ consecutively. Defant and Zheng extended this to a consecutive-pattern-avoiding stack-sorting map $SC_\sigma$, where the stack must always avoid a given permutation $\sigma$ consecutively. We address one of the main conjectures raised by Defant and Zheng in their dynamical approach to $SC_\sigma$. Specifically, we show that the periodic points of $SC_\sigma$ are precisely the permutations that consecutively avoid $\sigma$ and its reverse.
\end{abstract}

\section{Introduction}

In 1968, Knuth defined a stack-sorting ``machine" which operates on permutations \cite{knuth}. Knuth's machine stores a separate ``stack." At any point, a user can ``push" the next element of the permutation onto the stack or ``pop" the top element off the stack and into the output permutation. Knuth showed that a permutation can be sorted to the identity using the machine if and only if it avoids the pattern $231$. (We refer the reader to Section~\ref{sec:preliminaries} for formal definitions about pattern avoidance and stack-sorting.) In 1990, West defined a deterministic method of applying Knuth's ``machine": a map $s$ on permutations of length $n$ \cite{37}. To do so, West first imposed an additional condition on the stack: one cannot ``push" a larger element on top of a smaller element. West's map $s$ uses the machine to apply the operation ``push" whenever possible; otherwise, it uses the machine to apply the ``pop" operation. (See Figure~\ref{fig:machine} for an illustration and Figure~\ref{fig:west} for an example of $s$.) Notably, West's map $s$ also sorts every $231$-avoiding permutation to the identity. Recently, several mathematicians have explored generalizations of Knuth and West's ideas. Albert, Homberger, Pantone, Shar, and Vatter \cite{albert2018generating} defined $\mathcal{C}$-machines for a permutation class $\mathcal{C}$, in which the stack must always be in the same relative order as a member of $\mathcal{C}$. Cerbai, Claesson, and Ferrari \cite{cerbai2020stack} looked at $\mathcal{C}$-machines in the special case that $\mathcal{C} = \Av(\sigma)$, where $\Av(\sigma)$ is the set of permutations which avoid the pattern $\sigma$. They defined a generalization of $s$ with $\sigma$-avoiding stacks called $s_\sigma$. In this context, West's map is equivalent to $s_{21}$.

\begin{figure}
    \centering
    \includegraphics[width=250pt]{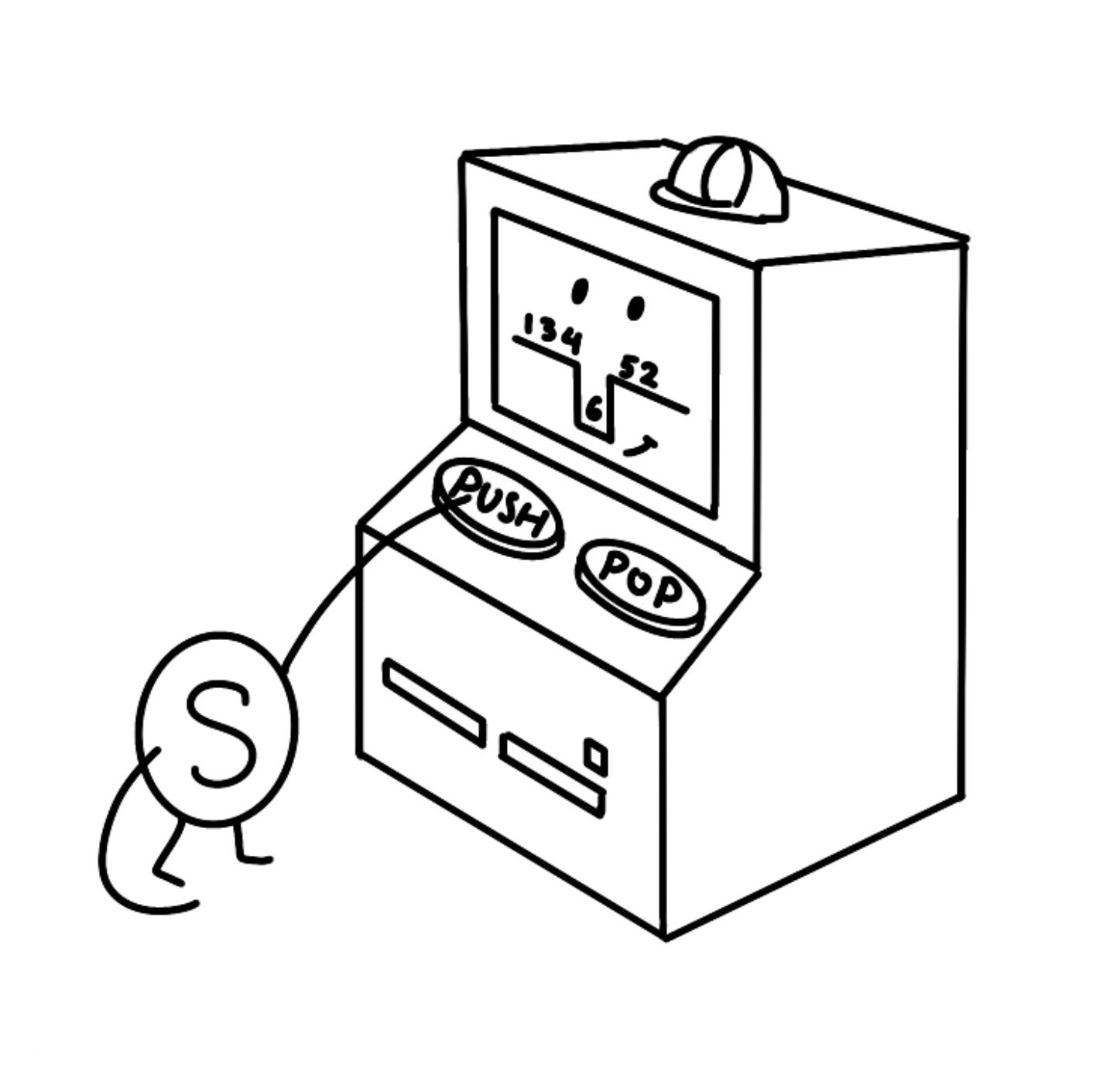}
    \vspace{-0.5cm}
    \caption{West's stack-sorting map $s$ operating Knuth's stack-sorting ``machine."}
    \label{fig:machine}
\end{figure}

In 2020, Defant and Zheng \cite{defant2021stack} introduced another variant of West's map: consecutive-pattern-avoiding stack-sorting maps $SC_\sigma$. Rather than requiring the stack to avoid a permutation $\sigma$, they instead required the stack to avoid occurrences of $\sigma$ that appear in consecutive indices. Note that a permutation avoids $21$ if and only if it avoids $21$ consecutively. Thus, West's map $s$ is equivalent to $s_{21}$ and $SC_{21}$, and serves as a ``base case" for both the pattern-avoiding and consecutive-pattern-avoiding generalizations.

In order to determine whether to ``push" or ``pop," the classical pattern-avoiding extension $s_\sigma$ requires knowledge of the whole stack. In contrast, the consecutive pattern-avoiding variant $SC_\sigma$ for $\sigma \in S_k$ only requires knowledge of the top $k-1$ entries of the stack. This is very similar to West's map, where each step only depends on the top element of the stack; $SC_\sigma$ preserves the inner-workings of West's original map. This paper focuses on the properties of the map $SC_\sigma$.

Many of the questions asked about stack-sorting maps concern how ``effectively" they sort a permutation. For example, which permutations are sorted to the identity? Which permutations can be sorted to the identity with two applications of the stack-sorting map? Given an arbitrary permutation, how many times must we apply the stack-sorting map to get the identity \cite{6, 7, 8, 13, 15, 17, 22, 37, 38}? Following this theme, Defant and Zheng studied the set $\text{Sort}_n(SC_\sigma)=SC_\sigma^{-1}(\Av(231))$; i.e., the set of permutations of length $n$ which sort to the identity after applying $SC_\sigma$ and then $s$ \cite{defant2021stack}. (Note that analogous questions were previously addressed in the classical pattern-avoiding case \cite{11}.)

Stack-sorting has also been explored from a dynamical perspective \cite{defant2021stack, 9, 14, 16, 18, 19}. For example, questions in this direction often concern the \textit{fertility} of a permutation $\sigma$, which is defined as the size of the preimage of $\sigma$ under the stack-sorting map \cite{16, defant2021stack}. Defant and Zheng were the first to apply this perspective to stack-sorting maps with pattern avoidance conditions \cite{defant2021stack}. The main goal of this paper is to prove \cite[Conjecture 8.1]{defant2021stack}, stated below. We write $\Av_n(\usigma, \urevsigma)$ for the set of length $n$ permutations which avoid $\sigma$ and $\text{rev}(\sigma)$ consecutively. See Sec~\ref{sec:preliminaries} for detailed definitions.

\begin{theorem}\label{thm:main}
    Let $\sigma \in S_k$ for $k \geq 3$. The periodic points of the map $SC_\sigma: S_n \rightarrow S_n$ are precisely the permutations in $\mathrm{Av}_n(\underline{\sigma}, \underline{\rev(\sigma)})$.
\end{theorem}

Defant and Zheng showed that this theorem holds for permutations $\sigma$ of length $3$ \cite{defant2021stack}. In Section~\ref{sec:main}, we present a proof in the general case for $\sigma$ of any length greater or equal to $3$.

Some work has been done about stacks which avoid several patterns classically \cite{3}, but the analog for consecutive-pattern-avoiding stacks has not been explored.  In particular, let $B$ be a finite set of permutations, and define the map $SC_{B}$, where the stack is required to avoid consecutive occurrences of every $\sigma \in B$. The following extension of Theorem~\ref{thm:main} remains an open question.

\begin{question}
    What are the periodic points of the map $SC_{B}$?
\end{question}

In Section~\ref{sec:counterex}, we provide a counterexample to a conjecture by Defant and Zheng about the number of iterations of $SC_{231}$ required to sort a permutation of length $n$ (stated below). We also suggest several future directions inspired by this conjecture.

\begin{conjecture}[{\cite[Conjecture 8.2]{defant2021stack}}]\label{conj:wrong}
Let $\pi \in S_n$ for $n \geq 3$. Then $SC_{231}^{2n-4}(\pi) \in \Av_n(132, 231)$. Furthermore, there exists $\tau \in S_n$ for which $SC_{231}^{2n-5}(\tau) \not\in \Av_n(132, 231)$.
\end{conjecture}

Several of Defant and Zheng's conjectures are still open. We refer the reader to \cite[Conjecture 8.3]{defant2021stack}, regarding the fertility of permutations in the consecutive-pattern-avoiding case, as well as \cite[Conjecture 5.1]{defant2021stack}, regarding the size of $\text{Sort}_n(SC_{231})$. Note that \cite[Conjecture 8.4]{defant2021stack} has been addressed in \cite{choi}.

\section{Preliminaries} \label{sec:preliminaries}

A \textit{permutation} $\pi$ is a rearrangement of the integers $1, 2, \ldots, n$, where the entry in index $i$ is denoted by $\pi(i)$. The \textit{conjugate} $\pi^c$ of a permutation $\pi$ is the permutation obtained by replacing each entry $\pi(i)$ in $\pi$ with $n-\pi(i)+1$. Denote the set of all permutations of $\{1, \ldots, n\}$ by $S_n$. Given any sequence $\tau$ of $n$ distinct integers, the standardization $\text{std}(\tau) \in S_n$ is the unique permutation such that $\text{std}(\tau)(i) < \text{std}(\tau)(j)$ if and only if $\tau(i) < \tau(j)$. In this case, we say that the elements of $\pi$ are in the same \textit{relative order} as the elements of $\tau$.

Given a sequence of distinct integers $\pi$ and a permutation $\sigma$, we say that $\pi$ \textit{contains an occurrence} of the \textit{pattern} $\sigma$ if there exists a subsequence of $\pi$ in the same relative order as $\sigma$. If $\pi$ does not contain an occurrence of $\sigma$, we say that $\pi$ \textit{avoids} $\sigma$. The set of permutations that avoid $\sigma$ is denoted by $\Av(\sigma)$.

When $\pi$ contains an occurrence of $\sigma$ which appears in consecutive indices of $\pi$, we say that $\pi$ contains a \textit{consecutive occurrence} of the pattern $\sigma$. For convenience, we distinguish consecutive occurrences of $\sigma$ with an underline. The set of all permutations which avoid avoid $\usigma$ is denoted $\Av(\usigma)$. For more background about classical and consecutive pattern occurrences, we refer the reader to \cite{6, 20, 27}.

We now introduce the stack-sorting machine. The machine is defined by two operations: ``push" and ``pop." ``Pushing" refers to removing the earliest remaining element of the input permutation and placing it on top of a ``stack"; ``popping" refers to removing the top element of the stack and placing it at the end of the output string. Notably, just like a literal stack of objects, no elements in the stack other than the one on top can be modified. West defined a deterministic map $s: S_n \rightarrow S_n$ that applies this machine to permutations. It begins by pushing the first element $\sigma(1)$ of a permutation $\sigma$ into the stack. At any given step, the machine ``pushes" the next element of the permutation onto the stack if it is smaller than the top element of the stack. Otherwise, it ``pops" the top element of the stack into the output permutation. In Figure~\ref{fig:west}, we demonstrate how the map $s$ acts on the permutation $143652$. The numbers on the right represent the \textit{input} permutation; those on the left represent the \textit{output} permutation; and those in the ``ditch" represent the \textit{stack}.
    \begin{figure}[h!]
        \centering
        \vspace{0.25cm}\includegraphics[width=450pt]{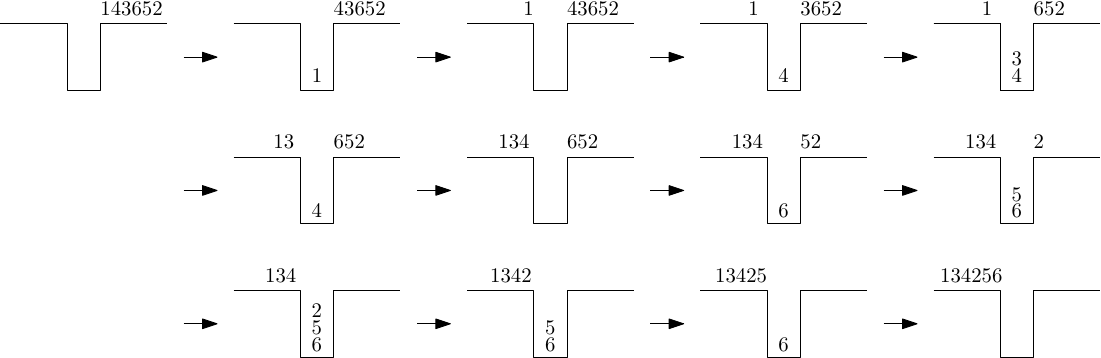}
        \vspace{0.15cm}
        \caption{Sorting 143652 with West's map $s$}
        \label{fig:west}
    \end{figure}

The pattern-avoiding stack-sorting map $s_{\sigma}: S_n \rightarrow S_n$ is defined similarly \cite{11}. (We will not study $s_{\sigma}$, but we include it here as relevant background.) At any step in the application of $s_\sigma$, the machine ``pushes" the next element of the permutation onto the stack, unless this causes an occurrence of $\sigma$ in the stack (read from top to bottom), in which case the machine ``pops." Note that the stack never contains an occurrence of $\sigma$. See Figure~\ref{fig:classical} for an example of the map $s_{231}$. To define the consecutive-pattern-avoiding stack-sorting map $SC_\sigma$, we simply replace $\sigma$ by $\usigma$ in the definition of $s_\sigma$. See Figure~\ref{fig:consecutive} for an example of the map $SC_{231}$.
    \begin{figure}
        \centering
        \vspace{0.25cm}\includegraphics[width=450pt]{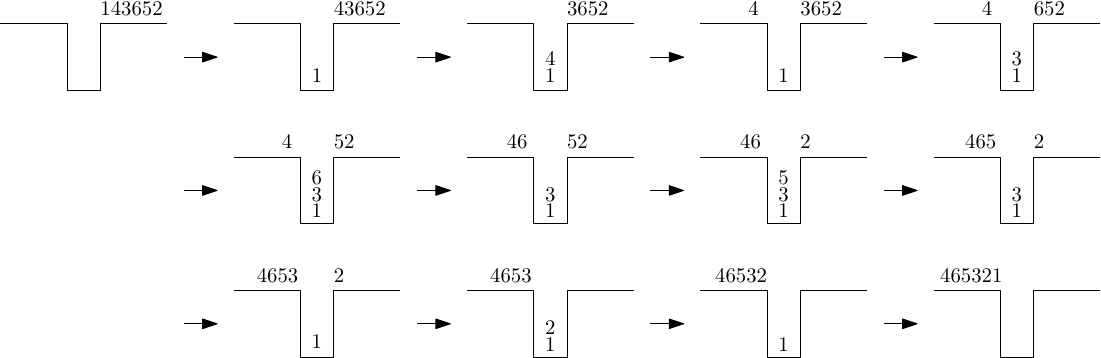}
        \vspace{0.15cm}
        \caption{Sorting 143652 with the pattern-avoiding stack-sorting map $s_{231}$.}
        \label{fig:classical}
    \end{figure}

    \begin{figure}
        \centering
        \vspace{0.25cm}\includegraphics[width=450pt]{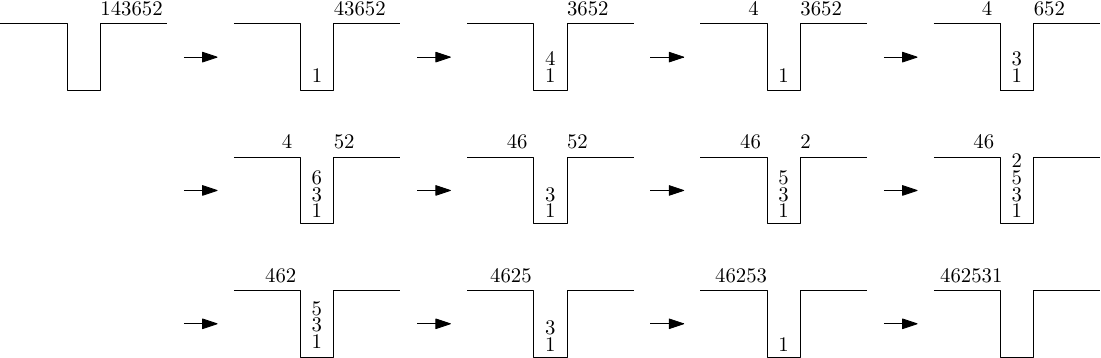}
        \vspace{0.15cm}
        \caption{Sorting 143652 with the consecutive-pattern-avoiding stack-sorting map $SC_{231}$.}
        \label{fig:consecutive}
    \end{figure}

During the application of a stack-sorting map, we say that an element is \textit{pre-popped} if it is popped before the final element of the permutation enters the stack, and \textit{post-popped} otherwise.

\section{Periodic Points of $SC_\sigma$}\label{sec:main}

This goal of this section is to prove Theorem~\ref{thm:main}. We begin by introducing the following notation. Suppose that we apply $SC_\sigma$ to a permutation $\pi$. Consider the step just before an element $e$ of $\pi$ is popped. Denote by $f(e)$ the sequence obtained by concatenating the stack (read from bottom to top) with the remaining segment of the input permutation. For example, in Figure~\ref{fig:consecutive}, $f(6)= 13652$ (see Step $5$).

\begin{lemma}\label{lem:2nd-to-last}
    Let $\sigma \in S_k$ for $k \geq 2$. Suppose that when $SC_\sigma$ is applied to $\pi$, an element $e$ is pre-popped from the stack. Then $f(e)$ contains an occurrence of $\urevsigma$ in which $e$ corresponds to $\sigma(2)$.
\end{lemma}
\begin{proof}
    Suppose that, at the step when $e$ is pre-popped, we instead apply the operation ``push." By definition of the map $SC_\sigma$, this would cause an occurrence of $\usigma$ in the stack. The result follows immediately. 
\end{proof}

\begin{lemma}\label{lem:n}
    Let $\pi \in S_n$, and suppose that we apply $SC_\sigma$ to $\pi$. If $\sigma(2) \neq k$, then the maximal element $n \in \pi$ is post-popped; likewise, if $\sigma(2) \neq 1$, then the minimal element $1 \in \pi$ is post-popped.
\end{lemma}

\begin{proof}
    Lemma~\ref{lem:2nd-to-last} implies that an element $e$ can only be pre-popped if it corresponds to $\sigma(2)$ in an occurrence of $\urevsigma$ in $f(e)$.
\end{proof}

We adopt the following notation for the position of points in a permutation $\pi$ relative to the position of $n$. Define \textit{the index of the entry j in $SC_\sigma^a(\pi)$} to be $$(-1)^{a}\left(SC_\sigma^a(\pi)^{-1}(j)-SC_\sigma^a(\pi)^{-1}(n)\right).$$ 

Note that in the permutation $\pi$, the entry $n$ is indexed by $0$, with positive indices to the right and negative indices to the left of $n$. When $\pi \in \Av(\usigma, \urevsigma)$, the permutation reverses each time we apply $SC_\sigma$. For the purposes of the proof of Theorem~\ref{thm:main}, we would like the values at each index to be preserved under $SC_\sigma$ when $\pi \in \Av(\usigma, \urevsigma)$. Thus, we include the constant $(-1)^{a}$. For the remainder of this section, when we discuss the indices of $SC_\sigma^a(\pi)$, we are always referring to this expression. 
When the permutation $\pi$ is clear from context, we denote \textit{the entry in index i of $SC_\sigma^a(\pi)$} by $e_i^a$.

\begin{proof}[Proof of Theorem~\ref{thm:main}.]

Fix $\sigma \in S_k$ for some $k \geq 3$. Assume first that $\pi \in \Av(\usigma, \urevsigma)$. Then it is clear that $SC_\sigma(\pi) = \rev(\pi)$ and $SC^2_\sigma(\pi)=\pi$. Thus, $\pi$ is a periodic point of $SC_\sigma$.

We now show the reverse direction; any periodic point $\pi$ of $SC_\sigma$ belongs to the set $\Av(\usigma, \urevsigma)$. Since $SC_{\sigma}(\pi) = (SC_{\sigma^c}(\pi^c))^c$, the periodic points of $SC_{\sigma^c}$ are precisely the complements of those of $SC_\sigma$. Thus, the theorem statement holds for $SC_\sigma$ if and only if it holds for $SC_{\sigma^c}$. Assume without loss of generality that $\sigma(2) \neq k$; otherwise, we may equivalently consider $\sigma^c$.

We will first show that no element of $\pi$ is pre-popped under arbitrarily many applications of $SC_\sigma$. Assume the contrary; denote by $m$ an index with minimal absolute value that is pre-popped after any number of applications of $SC_\sigma$. (Note that $m$ is not necessarily unique.)  By Lemma~\ref{lem:n}, we have that $m \neq 0$. Assume without loss of generality that $m > 0$; the proof is analogous when $m < 0$. Note that the elements in indices $0$ through $m-1$ (inclusive) are never pre-popped, so their values are constant under arbitrarily many applications of $SC_\sigma$. This means that the entry in index $m$ only changes when it is pre-popped.

There are two possible ways for the entry in index $m$ to be pre-popped. First, suppose that after $a$ applications of $SC_\sigma$, the index $m$ is to the right of the index $0$. Further suppose that applying $SC_\sigma$ again causes the entry $e_m^a$ in index $m$ to be pre-popped from the stack. By Lemma~\ref{lem:2nd-to-last}, $e_m^a$ must correspond to $\sigma(2)$ in an occurrence of $\urevsigma$ in $f(e_m^a)$. Furthermore, after the entry $e_m^a$ is pre-popped, the element corresponding to $\sigma(1)$ in $f(e_m^a)$ is pushed onto the stack and replaces it in index $m$ of the output permutation. Now suppose that after $a$ applications of $SC_\sigma$, the index $m$ is to the left of the index $0$ and applying $SC_\sigma$ again causes the entry $e_m^a$ in index $m$ to be pre-popped. As before, this means that $e_m^a$ corresponds to $\sigma(2)$ in an occurrence of $\urevsigma$ in $f(e_m^a)$. Furthermore, the top element of the stack after $e_m^a$ is pre-popped correspnds to $\sigma(3)$ in the occurrence. The elements in indices $m-1$ through $0$ inclusive are then pushed onto the stack. Thus, in the output permutation, the entry in index $m$ is the element corresponding to $\sigma(3)$ in $f(e_m^a)$.

We consider the following two cases.

\begin{itemize}
    \item \textbf{Case 1.} Suppose that the first three elements of $\sigma$ form an occurrence of $\underline{132}$ or $\underline{231}$ (resp. $\underline{213}$ or $\underline{312}$). Then in both possibilities described above, when the entry in index $m$ is pre-popped, it is replaced by a smaller (resp. larger) element. However, we assumed that $\pi$ is a periodic point, so this is impossible.
    \item \textbf{Case 2.} Otherwise, suppose that the first three elements of $\sigma$ form an occurrence of $\underline{123}$. (The proof is analogous for $\underline{321}$.) We consider two subcases.
    
    Suppose first that the entry in index $m$ of $\pi$ is greater than the entry in index $m-1$ of $\pi$; that is, $e_m > e_{m-1}$. We will argue by induction that this relation is preserved under applications of $SC_\sigma$. Assume that $e_m^{a-1}>e_{m-1}^{a-1}$ for some $a \in \mathbb{N}$; we aim to show that $e_m^a > e_{m-1}^a$. If the index $m$ is to the right of the index $0$ in $SC_\sigma^{a-1}(\pi)$, and we apply $SC_\sigma$ again, then $e_m^{a-1}$ enters the stack directly after the entry $e_{m-1}^{a-1}$. (By assumption, $e_{m-1}^{a-1}$ is not pre-popped.) Since $e_m^{a-1}>e_{m-1}^{a-1}$ but $\sigma(2)<\sigma(3)$, the element $e_m^{a-1}$ cannot correspond to $\sigma(2)$ in an occurrence of $\urevsigma$ in $f(e_{m}^{a-1})$. It follows from Lemma~\ref{lem:2nd-to-last} that the entry $e_m^{a-1}$ in index $m$ is not pre-popped. If the index $m$ is to the left of the index $0$ and applying $SC_\sigma$ again causes the entry in index $m$ to be pre-popped, then it is replaced by a larger entry. In either case, we have that $e_m^{a}>e_{m-1}^{a}$, as desired. We have also shown that if the entry in index $m$ is pre-popped, then it is replaced by a larger element. Thus, the value of the entry in index $m$ weakly increases under repeated applications of $SC_\sigma$. However, we assumed that the entry in index $m$ is pre-popped at least once, so its value is not constant. This is a contradiction as $\pi$ is periodic.

    Thus, we must have that $e_m < e_{m-1}$. Since the above argument holds for all periodic points, we must also have that $e_m^a < e_{m-1}^a$ for all $a \in \mathbb{N}_0$. Suppose that after $a$ applications of $SC_\sigma$, the index $m$ is to the left of the index $0$. Since $e_m^a < e_{m-1}^a$, the sequence $f(e_m^a)$ cannot contain an occurrence of $\urevsigma$ in which $e_{m-1}^a$ corresponds to to $\sigma(1)$ and $e_m^a$ corresponds to $\sigma(2)$. Thus, the entry $e_{m-1}^a$ enters the stack directly after $e_{m}^a$. Since $e_{m-1}^a$ is post-popped, so is $e_m^a$. If the index $m$ is to the right of the index $0$ and applying $SC_\sigma$ causes $e_m^a$ to be pre-popped, then it is replaced by a smaller entry in index $m$ of the output permutation. Thus, the value of the entry in index $m$ is weakly decreasing under repeated applications of $SC_\sigma$. Since we assumed that the entry in index $m$ is pre-popped at least once, its value is not constant. This is a contradiction as $\pi$ is periodic.
\end{itemize}

Thus, no entry in $\pi$ is pre-popped by any number of applications of $SC_\sigma$. Note that if $\pi$ contains an occurrence of $\urevsigma$, then an element must be pre-popped after one application of $SC_\sigma$; if $\pi$ contains an occurrence of $\usigma$, then $SC_\sigma(\pi) = \text{rev}(\pi)$ contains an occurrence of $\urevsigma$, so an element must be pre-popped after two applications of $SC_\sigma$. Both cases cause a contradiction, so $\pi \in \Av(\usigma, \urevsigma)$.
\end{proof} 

We conclude this section with the following consequence of Theorem~\ref{thm:main}.

\begin{corollary}
    Every periodic point of $SC_\sigma$ has period $2$.
\end{corollary}
\section{Permutations That Sort to Periodic Points in Maximal Time}\label{sec:counterex}

We restate \cite[Conjecture 8.2]{defant2021stack} and present a counterexample below.

\begin{conjecture}[Defant and Zheng \cite{defant2021stack}]\label{conj:wrong}
For any permutation $\pi$ of length $n \geq 3$, we have $SC_{231}^{2n-4}(\pi) \in \Av_n(132, 231)$. Also, for every $n \geq 3$, there exists $\tau \in S_n$ for which $SC_{231}^{2n-5}(\tau) \not\in \Av_n(132, 231)$.
\end{conjecture} %is it OK to state conjecture twice?

Note that $\Av_n(132, 231) = \Av_n(\underline{132}, \underline{231})$. Thus, Conjecture~\ref{conj:wrong} describes how many steps it takes to sort a permutation to a periodic point. Defant and Zheng verified that the statement holds for $n \leq 9$; the following counterexample has length $11$.

\begin{counterexample}
Consider the permutation $(4, 6, 8, 5, 11, 7, 2, 9, 10, 3, 1) \in S_{11}$. This permutation sorts to a periodic point after $19$ steps, as shown below.
\begin{align*} 
&\hspace{0.54cm}(11, 1, 3, 10, 9, 2, 7, 5, 8, 6, 4) \rightarrow (10, 7, 8, 4, 6, 5, 2, 9, 3, 1, 11) \rightarrow (6, 9, 11, 1, 3, 2, 5, 4, 8, 7, 10) \\
&\rightarrow (3, 5, 8, 10, 7, 4, 2, 1, 11, 9, 6) \rightarrow  (11, 9, 6, 1, 2, 4, 7, 10, 8, 5, 3) \rightarrow (10, 3, 5, 8, 7, 4, 2, 1, 6, 9, 11) \\
&\rightarrow (8, 11, 9, 6, 1, 2, 4, 7, 5, 3, 10)
\rightarrow (11, 7, 10, 3, 5, 4, 2, 1, 6, 9, 8)
\rightarrow (5, 9, 8, 6, 1, 2, 4, 3, 10, 7, 11) \\
&\rightarrow (9, 8, 4, 10, 11, 7, 3, 2, 1, 6, 5)
\rightarrow (6, 5, 1, 2, 3, 7, 11, 10, 4, 8, 9)
\rightarrow (11, 9, 8, 4, 10, 7, 3, 2, 1, 5, 6) \\
&\rightarrow (10, 6, 5, 1, 2, 3, 7, 4, 8, 9, 11)
\rightarrow (7, 11, 9, 8, 4, 3, 2, 1, 5, 6, 10)
\rightarrow (11, 9, 10, 6, 5, 1, 2, 3, 4, 8, 7) \\
&\rightarrow (8, 7, 4, 3, 2, 1, 5, 6, 9, 11, 10)
\rightarrow (10, 11, 9, 6, 5, 1, 2, 3, 4, 7, 8)
\rightarrow (8, 7, 4, 3, 2, 1, 5, 6, 9, 11, 10) \\
&\rightarrow (11, 10, 9, 6, 5, 1, 2, 3, 4, 7, 8)
\end{align*}
\end{counterexample}

In their conjecture, Defant and Zheng aimed to bound the number of applications of $SC_{231}$ necessary to sort a permutation to a periodic point. We restate this question as follows.

\begin{question}\label{q:n}
For any $n \in \mathbb{N}$, let $f(n)$ be the maximum number of applications of $SC_{231}$ necessary to sort a length $n$ permutation to an element of $\Av_n(132, 231)$. What can be said about $f(n)$?
\end{question}

It is not difficult to provide a bound of $O(n^2)$. We present the proof below.

\begin{theorem}
    For any $n \in \mathbb{N}$, we have $f(n) \leq (n-1)(n-2).$
\end{theorem}

\begin{proof}
    We adopt several ideas from the proof of \cite[Proposition~3.2]{defant2021stack}. By Lemma~\ref{lem:2nd-to-last}, when $SC_{231}$ is applied to a permutation $\pi$, the elements $1$ and $2$ of $\pi$ are post-popped. Let $D(\pi)$ be the number of entries between $1$ and $2$ in $\pi$. If $D(\pi) > 0$, there exists an occurrence of either $\underline{231}$ or $\underline{132}$ between $1$ and $2$ inclusive. Thus, applying $SC_{231}$ twice causes at least one element strictly between $1$ and $2$ to be pre-popped. Thus, $D(SC^2(\pi)) < D(\pi)$ \cite{defant2021stack}.

    Let $\pi \in S_n$. Denote by $\pi^{*}$ the result of deleting the element $1$ and normalizing the result.
    
    \begin{claim}\cite{defant2021stack}\label{claim} If $1$ and $2$ are consecutive in $\pi$, then $SC_{231}(\pi^{*})=SC_{231}(\pi)^{*}$ and $1$ and $2$ are consecutive in $SC_{231}(\pi)$.
    \end{claim}

    This is because, if $1$ and $2$ are consecutive, no occurrence of $\underline{231}$ can contain both $1$ and $2$. Thus, $1$ and $2$ enter and exit the stack consecutively. Furthermore, if either $1$ or $2$ is contained in an occurrence of $231$, then it must correspond to $1$.

    Note that any permutation of length $1$ or $2$ is a periodic point. Thus, the statement holds for $n < 3$. We proceed by induction. Suppose that $SC_{231}^{(n-2)(n-3)}(\pi) \in \Av(132, 231)$ for all $\pi \in S_{n-1}$. Consider $\pi \in S_n$; we will show that $SC_{231}^{(n-1)(n-2)}(\pi) \in \Av(132, 231)$. Since $D(SC^2(\pi)) < D(\pi)$, we have that after at most $2(n-2)$ applications of $SC_{231}$, the elements $1$ and $2$ are consecutive. Call the permutation $\tau=SC_{231}^{2(n-2)}(\pi)$. By the inductive hypothesis, $SC_{231}^{(n-2)(n-3)}(\tau^{*}) \in \Av(132, 231)$. Claim~\ref{claim} implies that $SC_{231}^{(n-2)(n-3)}(\tau^{*}) = SC_{231}^{(n-2)(n-3)}(\tau)^{*}$ and $1$ and $2$ are consecutive in $SC_{231}^{(n-2)(n-3)}(\tau)$. Thus, $SC_{231}^{(n-2)(n-3)}(\tau) \in \Av(132, 231)$, so it takes at most $2(n-2)+(n-2)(n-3) = (n-1)(n-2)$ applications of $SC_{231}$ to sort $\pi$ to a periodic point. This completes the inductive proof.
\end{proof}

We may also generalize Question~\ref{q:n} to arbitrary permutations $\sigma$.

\begin{question}
    Let $\sigma \in S_k$ for some $k \geq 3$. For any $n \in \mathbb{N}$, let $g(n, \sigma)$ be the maximum number of applications of $SC_\sigma$ necessary to sort a length $n$ permutation to an element of $\Av_n(\usigma, \urevsigma)$. What can be said about $g(n, \sigma)$?
\end{question}

\section*{Acknowledgements}
This research was conducted at the University of Minnesota Duluth REU, with support from Jane Street Capital, the
National Security Agency, and the National Science Foundation (Grants 2052036 and 2140043). First and foremost, we are grateful to Mitchell Lee for his crucial feedback throughout the research and editing processes. We would also like to thank Noah Kravitz and Colin Defant for carefully reviewing our proofs. Finally, we would like to thank Joe Gallian for extending an invitation to the Duluth REU, without which this research would not have been possible.

\newpage
\bibliography{references}
\bibliographystyle{plain}

$\\ \\ \\$
\end{document}